\documentclass[11pt]{article}
\usepackage{hyperref} 

\usepackage{amsmath}
\usepackage{amssymb}
\usepackage{amscd}
\usepackage{textcomp}

\usepackage{amsthm}
\theoremstyle{plain}
\newtheorem{theorem}{Theorem}[section]

\newtheorem{lemma}[theorem]{Lemma}

\newtheorem{corollary}[theorem]{Corollary}
\newtheorem{proposition}[theorem]{Proposition}

\theoremstyle{definition}
\newtheorem{definition}[theorem]{Definition}

\newtheorem{example}[theorem]{Example}
\newtheorem{remark}[theorem]{Remark}

\theoremstyle{remark}

\usepackage{xy}
\xyoption{all}

\newdir^{ (}{{}*!/-3pt/\dir^{(}}    
\newdir^{  }{{}*!/-3pt/\dir^{}}    
\newdir_{ (}{{}*!/-3pt/\dir_{(}}    
\newdir_{  }{{}*!/-3pt/\dir_{}}

\newcounter{zahl}

\def\theenumi{(\alph{enumi})}

\def\p@enumii{\theenumi}

\newcommand{\DS}{\displaystyle}

\newcommand{\cA}{\mathcal{A}}

\newcommand{\cF}{\mathcal{F}}
\newcommand{\cG}{\mathcal{G}}

\newcommand{\cO}{\mathcal{O}}

\DeclareMathOperator{\Spec}{Spec}

\DeclareMathOperator{\charakt}{char}

\DeclareMathOperator{\rk}{rk}

\renewcommand{\phi}{\varphi}
\renewcommand{\epsilon}{\varepsilon}

\usepackage{amsfonts}
\newcommand{\BOne} {{\mathchoice{\hbox{\rm1\kern-2.7pt l\kern.9pt}}
                              {\hbox{\rm1\kern-2.7pt l\kern.9pt}}
                              {\hbox{\scriptsize\rm1\kern-2.3pt l\kern.4pt}}
                              {\hbox{\scriptsize\rm1\kern-2.4pt l\kern.5pt}}}}

\newcommand{\BD}{{\mathbb{D}}}

\newcommand{\BP}{{\mathbb{P}}}

\newcommand{\BZ}{{\mathbb{Z}}}

\newcommand{\CA}{{\cal{A}}}

\newcommand{\CF}{{\cal{F}}}
\newcommand{\CG}{{\cal{G}}}

\newcommand{\CO}{{\cal{O}}}
\newcommand{\CP}{{\cal{P}}}

\newcommand{\CT}{{\cal{T}}}

\let\setminus\smallsetminus

\newcommand{\ol}[1]{{\overline{#1}}}

\usepackage{ifthen}

\newcommand{\invlim}[1][]{\ifthenelse{\equal{#1}{}}
{\DS \lim_{\longleftarrow}}
{\DS \lim_{\underset{#1}{\longleftarrow}}}
}

\newcommand{\dirlim}[1][]{\ifthenelse{\equal{#1}{}}
{\DS \lim_{\longrightarrow}}
{\DS \lim_{\underset{#1}{\longrightarrow}}}
}

\newcommand{\dotBD}{\vbox{\hbox{\kern2pt\bf.}\vskip-4.5pt\hbox{$\BD$}}}

\usepackage{color}
\newcounter{commentcounter}

\def\?{\ 
{\bf\color{red}???}\ 
\immediate\write16{}
\immediate\write16{Warning: There was still a question mark . . . }
\immediate\write16{}}

\newbox\mybox
\def\arrover#1{\mathrel{
       \setbox\mybox=\hbox spread 1.4em{\hfil$\scriptstyle#1$\hfil}
       \vbox{\offinterlineskip\copy\mybox
             \hbox to\wd\mybox{\rightarrowfill}}}}

\begin{document}

\author{Somayeh Habibi and M. E. Arasteh Rad}

\title{On The Motive of G-bundles}

\maketitle

\begin{abstract}

Let $G$ be a reductive algebraic group over a perfect field $k$ and $\cG$ a $G$-bundle over a scheme $X/k$. The main aim of this article is to study the motive associated with $\cG$, inside the Veovodsky Motivic categories. We consider the case that $\charakt k=0$ (resp. $\charakt k\geq 0$), the motive associated to $X$ is geometrically mixed Tate (resp. geometrically cellular) and $\cG$ is locally trivial for the Zariski (resp. \'etale) topology on $X$ and show that the motive of $\cG$ is geometrically mixed Tate. Moreover for a general $X$ we construct a nested filtration on the motive associated to $\cG$ in terms of weight polytopes. Along the way we give some applications and examples.

\end{abstract}

\bigskip

%
%

\bigskip
\section{Introduction}

Let $k$ be a perfect field and let $X$ be a scheme of finite type over $k$. In this paper we study motives of certain fiber bundles over $X$. 
Specially we are interested in the motives of $G$-bundles over a base scheme $X$ for a connected reductive group $G$ over $k$.
\newline Let $\cG$ be a $G$-bundle as above. When $\charakt k=0$ and the motive associated to $X$ is geometrically mixed Tate, we show that the motive of any Zariski locally trivial $G$-bundle $\cG$ is also geometrically mixed Tate. In addition, when we restrict ourselves to the case that $X$ is geometrically cellular we may treat the positive characteristic case as well and moreover we are not required to assume that $\cG$ is Zariski locally trivial. Although we are mainly interested in the case that the motive associated to $X$ is mixed Tate, our method will also produce a nested filtration on the motive of $\cG$ over a general base. In particular we apply this to the case that $X$ is a smooth projective curve.
\newline
As a consequence of this discussion we observe that the motive of a split reductive group is mixed Tate. Using a recent result of McNinch \cite{McN} we will see that a similar fact holds for certain Parahoric group schemes. The former case has already been studied by Biglari, see \cite{Big} where he makes the computations with rational coefficients. Moreover A. Huber and B. Kahn have studied the motive of split reductive groups using the theory of slice filtration, see \cite{H-K}. Let us mention that in \cite[section 8]{H-K} they also produce a filtration for split torus bundles. 
\newline
In this paper our approach is essentially based on several geometric observations, namely we introduce a motivic version of the Leray-Hirsch theorem and we implement the combinatorial tools provided by the theory of wonderful compactification of semi-simple algebraic groups of adjoint type, see \cite{DeConciniProcesi}. 
 
Let us have a brief look at the content of this article. In section \ref{Notation and Conventions} we fix notation and conventions. In section \ref{Motive of cellular fibrations} we introduce the notion of motivic relative cellular varieties and prove some of their basic properties. Furthermore we implement a result of B. Kahn which describes the (geometric) motive of cellular varieties (see \cite{BKahn}) to establish a motivic version of Leray-Hirsch theorem for cellular fibrations. In the next section we recall some results about the geometry of wonderful compactifiation of a reductive group of adjoint type. Subsequently we see that $G\times G$-orbit closures are (motivic) cellular. Using this we study the case that $\charakt k=0$ and $X$ is geometrically mixed Tate (resp. $\charakt k >0$ and $X$ is geometrically cellular). In addition we introduce particular applications. Finally in section \ref{Filtration on the motive of G-bundles} we discuss the case that the base scheme $X$ is not necessarily mixed Tate and produce a filtration on the motive associated to $\cG$. Meanwhile we apply this to the case where $X$ is a projective curve.

\bigskip
\noindent
\emph{Acknowledgement.} We are grateful to L. Barbieri Viale and B.Kahn for their helps and comments on the earlier draft of this article. We thank J. Bagalkote  for editing and to J. Scholbach for the helpful discussions regarding this work.

\tableofcontents 

\section{Notation and Conventions}\label{Notation and Conventions}

Throughout this article we assume that $k$ is a perfect field. We denote by $Sch_k$ (resp. $Sm_k$) the category of schemes (resp. smooth schemes) of finite type over $k$. 

For $X$ in $\CO b(Sch_k)$, let $CH_i(X)$ and $CH^i(X)$ denote Fulton's $i$-th Chow groups and let $CH_\ast(X):=\oplus_i CH_i(X)$ (resp. $CH^\ast(X):=\oplus_i CH^i(X)$).  

We denote by $Sch_k^{fr}$  (resp. $Sm_k^{fr}$) the full subcategory of $Sch_k$ (resp. $Sm_k$) consisting of those $X \in \CO b(Sch_k)$ (resp. $X \in \CO b(Sm_k)$) that $CH_\ast(X)$ is free of finite rank over $\BZ$. 

\begin{remark}\label{RemBassConjecture}
The category $Sm_k^{fr}$ need not be a tensor category. Even after passing to the coefficients in $\mathbb{Q}$, it is not obvious to the authors that whether the full subcategory of $Sm_k$ consisting of objects $X$ with $rk_{\mathbb{Q}}K_0(X)< \infty$ is a tensor category or not. However if one assumes the Bass conjecture, then this is a trivial consequence. 
\end{remark}

To denote the motivic categories over $k$, such as $DM_{gm}^{eff}(k)$, $DM_{-}^{eff}(k)$, $DM_{-}^{eff}(k) \otimes\mathbb{Q}$ and etc., and the functors $M:Sch_k\rightarrow DM_{gm}^{eff}(k)$ and $M^c:Sch_k\rightarrow DM_{gm}^{eff}(k)$ constructed by Voevodsky, we use the same notation that was introduced by him in \cite{VV}. These constructions were recently developed by  Cisincki and Deglise. See \cite{CD} and Voevodsky \cite{VVI}, where they construct the triangulated category of motives over a general base scheme $S$.

\begin{definition}\label{DefMixedTate}
The thick subcategory of $DM_{gm}^{eff}(k)$, generated by $\mathbb{Z}(0)$ and $\mathbb{Z}(1)$ is called \emph{the category of mixed Tate motives} and we denote it by $TDM_{gm}^{eff}(k)$. Any object of $TDM_{gm}^{eff}(k)$ is called a
 \emph{mixed Tate motive}. A motive $M$ is \emph{geometrically mixed Tate} if it becomes mixed Tate over $\overline{k}$.
\end{definition}

We simply denote $A\to B \to C$ to denote a distinguished triangle $A \to B \to C \to A[1]$ in either of the above categories. \\

\textbf{CAUTION}: Throughout this article we either assume that $k$ admits resolution of singularities or coefficients in $\mathbb{Q}$. 

For the definition of the geometric motives with compact support in positive characteristic we refer to \cite{H-K}.

Let us now move to the algebraic group theory side. Let $G$ be a reductive group over $k$. We denote by $G^s$ the semi-simple quotient of $G$ and by $G^{ad}$ the adjoint group of $G$. 

Consider an algebraic closure $\bar{k}$ of $k$. Fix a maximal torus $T$ in $G_{\bar{k}}$ and a Borel subgroup $B$ of $G_{\bar{k}}$ that contains $T$. Let $X^\ast(T)$ (resp. $X_\ast(T)$) denote the group of cocharacters (resp. characters) of $G$. 
Let $\Phi:=\Phi(T,G_{\bar{k}})$ be the associated root system and $\Delta \subseteq \Phi(T,G_{\bar{k}})$ be a system of simple roots (i.e. a subset of $\Phi$ which form a basis for $Lie(G_{\bar{k}})$ such that any root $\beta \in \Phi$ can be represented as a sum $\beta =\sum_{\alpha \in \Delta}m_{\alpha} \alpha$, with $m_{\alpha}$ all non-negative or all non-positive integral coefficients). Let $W:=W(T,G_{\bar{k}})$ and $l:W\rightarrow \mathbb{Z}_+$ denote respectively the corresponding Weyl group and the usual length function on $W$.  For any subset $I \subseteq \Delta$ we set $\Phi_I$ to be the subset of $\Phi$ spanned by $I$. Furthermore for $u\in W$, $I_u$ will denote the set consisting of those elements of $\Delta$ that do not occur in the shortest expression $u$. We denote by $W_I$ the subgroup of the Weyl group $W$ generated by the reflections associated with the elements of $\Phi_I$. Let $W^I$ denotes a set of representative for $W/W_I$ with minimal length. Notice that any parabolic subgroup $P$ of $G$ is conjugate with a standard Parabolic subgroup, i.e. to a group of the form $P_I:=BW_IB$.  Finally we denote by $\mathcal{B}_G:=\mathcal{B}_G(B,T)$ the associated Bruhat-Tits building.\\

Let $Y$ be a variety with left $G$-action. To a $G$-bundle $\CG$ on $X$ one associates a fibration $\CG\times^G Y$ with fibre $Y$ over $X$, defined by the following quotient
$$
\CG\times Y\big{\slash}\sim,
$$
here $(x,y)\sim (xg,g^{-1}y)$ for every $g\in G$.

%
%

\section{Motive of cellular fibrations} \label{Motive of cellular fibrations}
\setcounter{equation}{0}

First we introduce the notion of motivic relatively cellular. Notice that this notion is slightly weaker than the geometric notion of relatively cellular introduced by Chernousov, Gille, Merkurjev \cite{CGM} and also Karpenco \cite{Kar}.

\begin{definition}\label{relativecellular1}
A scheme $X \in \CO b(Sch_k)$ is called \emph{motivic relatively cellular with respect to} the functor $M^c(-)$ if it admits a filtration by its closed subschemes:
$$
\emptyset= X_{-1} \subset X_0 \subset ... \subset X_n=X 
$$ 
together with flat equidimensional morphisms $p_i: U_i:=X_i \setminus  X_{i-1} \rightarrow Y_i$ of relative dimension $d_i$, such that the induced morphisms $p_i^\ast:M^c(Y_i)(d_i)[2d_i] \rightarrow M(U_i)$ are isomorphisms in $DM_{gm}^{eff}(k)$. Here $Y_i$ is smooth proper scheme for all $1 \leq i \leq n$. Moreover we say that $X$ is cellular if $p_i$ is affine bundle and $Y_i =\Spec k$ for $0 \leq i \leq n$. 
\end{definition}

\begin{proposition}\label{freechow}
Suppose $k$ admits resolution of singularities. Assume that $X\in \cO b(Sch_k)$ is equidimensional of dimension $n$, which admits a filtration as in the definition \ref{relativecellular1}. Then we have the following decomposition
$$
M^c(X)=\bigoplus_i M^c(Y_i)(d_i)[2d_i].
$$
\end{proposition}

\begin{proof}
We prove by induction on $\dim X$. Consider the following distinguished triangle
$$
M^c(X_{j-1}) \rightarrow M^c(X_j) \xrightarrow{\;g_j} M^c(U_j) \rightarrow M^c(X_{j-1})[1].
$$
Take the closure of the graph of $p_j : U_j \rightarrow Y_j$ in $X_j \times Y_j$. This defines a cycle in $CH_{dim~X_j}(X_j\times Y_j)$ and since $Y_j$ is smooth this gives a morphism 
$$\gamma_j : M^c(Y_j)(d_j)[2d_j] \rightarrow M^c(X_j),$$
 by \cite[Chap. 5, Thm. 4.2.2.3) and Prop. 4.2.3]{VV}, such that $g_j \circ \gamma_j =p_j^\ast$. Thus the above distinguished triangle splits and hence we conclude by induction hypothesis.
\end{proof}

\begin{corollary}
Keep the notation and the assumptions of the above proposition.
Assume that each $Y_i$ belongs to $\CO b(Sm^{fr})$, then $X \in \CO b(Sch_k^{fr})$.
\end{corollary}
\begin{proof}
After apply the functor $Hom(\BZ(i)[2i],-)$ to the decomposition $M^c(X)=\bigoplus_i M^c(Y_i)(d_i)[2d_i]$, we obtained in the above proposition, the corollary follows from \cite[Proposition 19.18]{MVW}.
\end{proof}
\begin{remark}\label{relativecellular2}
Note that one can define a variant of the definition \ref{relativecellular1} with respect to the functor $M(-)$. In this case one has to replace $p_i^\ast$ by $p_{i_{\ast}}$ and it is not necessary to assume that $p_i$'s are flat. With this definition it is not hard to see that a variant of the proposition \ref{freechow} holds after imposing some additional condition. Indeed to apply Gysin triangle we have to assume that all $X_i$s that appear in the filtration of $X$ are smooth. Note that in this case we don't need to assume $k$ admits resolution of singularities. The proof goes similar to the proof of proposition \ref{freechow}.

\end{remark}

\begin{remark}\label{puretate}
Assume that $X$ is a motivic relatively cellular scheme, such that $Y_i$ is pure Tate for every $1 \leq i \leq n$. Then using noetherian induction and gysin triangle one can show that $X$ is pure Tate. 
\end{remark}
Let $Ab$ be the category of abelian groups. Let us recall that there is a fully faithful tensor triangulated functor $i: D_f^b(Ab) \rightarrow DM_{gm}^{eff}(k)$, where  $D_f^b(Ab)$ is the full subcategory of the bounded drived category $D^b(Ab)$, consisting objects with finitely generated cohomology groups, see \cite[Prop. 4.5.]{H-K}.
\newline

\begin{proposition}\label{MotivewcsofCellularVariety}
Assume that $\charakt k=0$. For a cellular variety $X \in \cO b(Sch_k)$, there is a canonical isomorphism $$
\coprod_{p\geq0} CH_p(X) \otimes \mathbb{Z}(p)[2p] \rightarrow M^c(X),
$$ 

which is functorial both with respect to proper or flat equidimentional morphisms.
\end{proposition}
\begin{proof}
C.f. \cite[Prop. 3.4.]{BKahn}.
\end{proof}

\begin{corollary}\label{Bruno Kahn result}

Let $X$ be as above. Assume further that it is equidimentional and smooth. Then there is a natural isomorphism in $DM_{gm}^{eff}(k)$:
$$
\coprod_{p \geqslant 0}CH^p(X)^{\vee} \otimes \mathbb{Z}(p)[2p] \rightarrow M(X),
$$
 where $CH^p(X)^{\vee}$ denotes the dual $\mathbb{Z}$-module.
\end{corollary}

\begin{proof} 
C.f. \cite[Cor. 3.5]{BKahn}.
\end{proof}

The most famous examples of cellular varieties are in fact generalized flag varieties. Let us state the following easy consequence of the above proposition applied to these particular examples.
 
\begin{corollary}\label{motive-general-flag-var}
(motive of a generalized flag variety)Let $G$ be a split reductive group, and let $P$ be a parabolic subgroup of $G$ which is conjugate with a standard parabolic subgroup $P_I$. Then there is an isomorphism 
$$
M(G/P)\cong \coprod_{w\in W^I} \mathbb{Z}(l(\omega))[2l(w)],
$$ 
in particular $G/P$ is pure Tate.
\end{corollary}

\begin{proof}
The decomposition $G=\coprod_{w\in W^I} BwP$ induces a cell decomposition $G/P\cong G/P_I=\coprod_{w\in W^I} X_w$, where $X_w\cong\mathbb{A}^{l(w)}$. Then $CH_\ast(G/P)$ is generated by the cycles $[\overline{X_w}]$ and thus we may conclude by the properness of $G/P$ and proposition \ref{MotivewcsofCellularVariety}.  
\end{proof}

In the rest of this section we are going to compute motives of certain fiber bundles. Recall that the naive version of \emph{Leray-Hirsch theorem} does not hold for the chow functor. One way to tackle the problem in the algebraic set-up is to impose some stronger conditions on the fiber. For instance one has to assume that the fiber admits cell decomposition and satisfies Poincar\'e duality (i.e. the intersection pairings $CH_p(F)\otimes CH^p(F)\rightarrow CH_0(F)$ are perfect parings, note that this is automatic when $\charakt k=0$). Let $f: \Gamma \rightarrow X$ be a smooth proper morphism that is locally trivial for the Zariski topology, with fiber $F$ which satisfies the above conditions. Let also $\zeta_1,...,\zeta_m$ be homogeneous elements of $CH^\ast(X)$ whose restriction to any fiber form a basis of its chow group over $\mathbb{Z}$. Then the Leray-Hirsch theorem for chow groups says that the homomorphism
$$
\varphi : \oplus_{i=1}^m CH_\ast(X) \rightarrow CH_\ast({\Gamma})~~~~~~,~\varphi(\oplus \alpha_i)= \Sigma \zeta_i \cap f^\ast\alpha_i
$$
 is an isomorphism. When $X$ is non-singular, it means that $\zeta_i$ form a free basis for $CH^\ast(\Gamma)$ as a $CH^\ast(X)-module$. For the proof we refer to \cite[appendix C]{Col-Ful}
\newline
Let us now state the \emph{motivic version of the Leray-Hirsch theorem}.
 
\begin{theorem}\label{Leray Hirsch for Voevodsky motives}
Let $X$ be a smooth irreducible variety over a field $k$ of characteristic $0$. Let $\pi:\Gamma \rightarrow X$ be a proper smooth locally trivial (for Zariski topology) fibration with fiber $F$.  Furthermore assume that $F$ is cellular. Then one has an isomorphism in $DM_{gm}^{eff}(k)$
$$
M(\Gamma)\cong \coprod_{p \geqslant 0}CH_p(F) \otimes M(X)(p)[2p].
$$  
\end{theorem}

\begin{proof}
Take a set of homogeneous elements $\{\zeta_{i,p}\}_{i,p}$ of $CH^\ast(\Gamma)$ such that for any $p$ the restrictions of $\{\zeta_{i,p}\}_i$ to any fiber $\Gamma_x\cong F$ form a basis for $CH^p(\Gamma_x)$. Notice that since $X$ is irreducible, it is enough that the restrictions of the $\zeta_i$'s generate $CH_\ast
(\Gamma_x)$ for the fiber over a particular $x$. 

By the theorems 14.16 and 19.1 of \cite{MVW}, for each $i$, $\zeta_{i,p}$ defines a morphism $M(\Gamma) \rightarrow \mathbb{Z}(p)[2p]$. Summing up all these morphisms and taking dual, by Poincar\'e duality we get the following morphism  
$$
\varphi: M(\Gamma)\rightarrow\bigoplus_p CH_p(F)\otimes \mathbb{Z}(p)[2p].
$$

Composing $M(\Delta): M(\Gamma)\rightarrow M(\Gamma\times \Gamma)\cong M(\Gamma) \otimes M(\Gamma)$ that is induced by the diagonal map $\Delta: \Gamma\times \Gamma \rightarrow \Gamma$, with $M(\pi)\otimes \varphi$ we obtain a morphism 
$$
M(\Gamma)\rightarrow \bigoplus_p CH^p(F)\otimes M(X)(p)[2p].$$

Now take a covering $\{U_i\}$ of $X$ that trivializes $\Gamma$. The theorem then follows from Mayer-Vietoris triangle, Kunneth formula \cite[Prop.4.1.7]{VV}, corollary \ref{Bruno Kahn result} and \cite[theorem 4.3.7, 3)]{VV}.  

\end{proof}

%
%

\section{Geometric Motives of G-bundles} \label{SectMotiveofG-bundles}

In this section we study motives of $G$-bundles over geometrically cellular or even geometrically mixed Tate variety $X$. Let us first recall the following result of A. Huber and B. Kahn.

\begin{proposition}\label{geometricmixedtate}

An object $M \in DM_{gm}^{eff}(k)$ is geometrically mixed Tate if and only if there is a finite
separable extension $E$ of $k$ such that the restriction of $M$ to $DM_{gm}^{eff}(E)$
is mixed Tate.

\end{proposition}

\begin{proof}
c.f. \cite[Proposition 5.3]{H-K}.
\end{proof}

\begin{definition} \label{DefMixedArrangement}
Let $X \in \cO b(Sch_k)$. We say that $X$ is \emph{mixed Tate} if the associated motive $M^c(X)$ is an object of the subcategory of mixed Tate motives $TDM_{gm}^{eff}(k)$. Let $\{ X_i\}_{i=1}^n$ be the set of irreducible components of  $X$. We call $X$ a \textit{configuration of mixed Tate varieties} if
\begin{enumerate}
\item[i)] $X_i$ is mixed Tate for $1 \leq i \leq n$, and
\item[ii)] Union of the elements of any arbitrary subset of $\{X_{ij} := X_i \cap X_j\}_{i \neq j}$ is a configuration of mixed Tate varieties or is empty.
\end{enumerate} 
\end{definition}

\begin{lemma}\label{LemMixedArrangement}
The motive of every configuration of mixed Tate varieties is mixed Tate.
\end{lemma}
\begin{proof}
We prove by induction on r, the dimension of mixed Tate configuration. The statement is obvious for $r=0$. Suppose that the lemma holds for all mixed Tate configurations of dimension $r<m$. Let $X=X_1 \cup \dots \cup X_n$ be a configuration of mixed Tate varieties of dimension m, where $X_i$s are its irreducible components. For inclusion  $\bigcup_{i \neq j}X_{ij} \subset \bigcup_{i=1}^n X_i$, we have the following induced localization distinguished triangle: 
$$
M^c(\bigcup_{i \neq j}X_{ij})\rightarrow M^c(X_1 \cup...\cup X_n) \rightarrow M^c(\bigcup_{i=1}^n X_i \setminus\bigcup_{i \neq j}X_{ij}) \rightarrow M^c(\bigcup_{i \neq j}X_{ij})[1].
$$
By the induction assumption, $ M^c(\bigcup_{i \neq j}X_{ij})$ is mixed Tate.
\newline
On the other hand we have:
$$
M^c(\bigcup_{i=1}^n (X_i\setminus\bigcup_{i \neq j}X_{ij}))= \bigoplus_{i=1}^n M^c(X_i\setminus\bigcup_{i \neq j}X_{ij}).
$$
It only remains to show that for every $i$, $M^c(X_i\setminus\bigcup_{i \neq j}X_{ij})$ is mixed Tate. To see this, for a given $i$ consider the following distinguished triangle:
$$
M^c(\bigcup_{j \neq i}X_{ij}) \rightarrow M^c(X_i) \rightarrow M^c(X_i\setminus\bigcup_{j \neq i}X_{ij}) \rightarrow M^c(\bigcup_{j \neq i}X_{ij})[1].
$$
Notice that $M^c(\bigcup_{i \neq j}X_{ij})$ is mixed Tate by induction hypothesis. 
\end{proof}

-- \emph{Wonderful Compactification} In \cite{Con-Pro} De Concini and Procesi have introduced the wonderful compactification of a symmetric space. In particular their method produces a smooth canonic compactification $\overline{G}$ of an algebraic group $G$ of adjoint type. Note that in \cite{Con-Pro} they study only the case that the group $G$ is defined over $\mathbb{C}$. Most of the theory carries over for any algebraically closed field of arbitrary characteristic. however there are some subtleties in positive characteristic which we mention later.
\newline 
As a feature of this compactification there is a natural $G \times G$-action on $\overline{G}$, and the arrangement of the orbits can be explained by the associated weight polytope. Let us briefly recall some facts about the construction of $\overline{G}$ and the geometry of its $G \times G$-orbits and their closure.

Let $\rho_{\lambda}:G \rightarrow GL(V_{\lambda})$ be an irreducible faithful representation of $G$ with strictly dominant highest weight $\lambda$. We define the compactification $X_{\lambda}$ of $G$ as follows
$$
X_{\lambda}= \overline{\mathbb{P}(\rho_{\lambda}(G))},
$$  
where the closure is taken inside $\mathbb{P}(End(V_{\lambda}))$.
\newline
It is verified in \cite{DeConciniProcesi} that when $G$ is of adjoint type, $X_{\lambda}$ is smooth and independent of the choice of the highest weight. This compactification is called \emph{wonderful compactification} and we denote it by $\overline{G}$.  

\bigskip

The following proposition explains the geometry of the wonderful compactification and the closures of its $G \times G$-orbits. Furthermore it provides an effective method to compute their cohomologies. Consider the one-to-one correspondence between polytopes and fans, which associates to a polytope its normal fan. Let $\CP_C$ denote the polytope associated to the fan of Weyl chambers and their faces.

\begin{proposition}\label{TheoWnderfulCompProperties} 
Keep the above notation, we have the following statements:

\begin{enumerate}
\item[a)] There is a one-to-one correspondence between the $G \times G$-orbits of $\ol{G}$ and the orbits of the action of the Weyl group on the faces of the polytope $\CP_C$, which preserves the incidence relation among orbits (i.e. consider the faces $\cF_1 \subseteq \cF_2$ of the polytope $\CP_C$, the orbit corresponds to the face $\cF_1$ is contained in the closure of the orbit which corresponds to $\cF_2$).   

\item[b)] Let $I \subset \Delta$ and $\mathcal{F}=\mathcal{F}_I$ the associated face of $\CP_C$. Let $D_{\mathcal{F}}$ be the closure of the orbit corresponding to the face $\mathcal{F}$. Then $D_{\mathcal{F}}=\sqcup_{\alpha\in W\times W}C_{\mathcal{F},\alpha}$, such that for each $\alpha:=(u,v)$ there is a bijective morphism
$$
\mathbb{A}^{n_{{\mathcal{F}},\alpha}}\rightarrow C_{\mathcal{F},\alpha},
$$  
where $n_{\mathcal{F},\alpha}=l(w_0)-l(u)+\mid I\cap I_u\mid +l(v)$ and $w_0$ denotes the longest element of the Weyl group. In particular when $\charakt k=0$ (resp. $\charakt > 0$) $D_{\mathcal{F}}$ is cellular (resp. motivic cellular).
\item[c)] $\overline{G} \setminus G$ is a normal crossing divisor, and its irreducible components form a mixed Tate configuration.   
\end{enumerate}
\end{proposition} 

\begin{proof}
For the proof of a) we refer to \cite[Prop.8]{Tima}. The existence of the bijective morphism in part b) is the main result of Renner in \cite{Renner}. The fact that $D_{\cF}$ is cellular in characteristic zero follows from Zariski main theorem. In positive characteristic this follows from the fact that any universal topological homeomorphism induces isomorphism of the associated h-sheaves, see \cite[Prop. 3.2.5]{VVII}. Finally c) follows from a), b) and remark \ref{puretate}.
\end{proof}

\begin{theorem}\label{motive-of-G-bundle1}
Assume that $\charakt k=0$. Let $G$ be a connected reductive group over $k$. Let $\cG$ be a $G$-bundle over an irreducible variety $X\in \CO b(Sm_k)$. Suppose that $\CG$ is locally trivial for Zariski topology and $X$ is geometrically mixed Tate, then $M(\cG)$ is also geometrically mixed Tate.
\end{theorem}

\begin{proof}
We may assume the base field $k$ is algebrically closed.
Let us first assume that $G$ is a semisimple group of adjoint type. 
 Then $G$ admits a wonderful compactification $\overline{G}$ which is smooth. By construction, there is a $(G \times G)$-action on $\overline{G}$. Let $G$ act on $\overline{G}$ via the first factor and consider the $\overline{G}$-fibration  $\overline{\mathcal{G}}:= \mathcal{G} \times _{G}\overline{G}$ over $X$. Clearly we have the open immersion $\mathcal{G} \hookrightarrow \overline{\mathcal{G}}$ of varieties over $X$. So we get the following generalized Gysin distinguished triangle:
$$
M(\mathcal{G}) \rightarrow M(\overline{\mathcal{G}}) \rightarrow M^c(\overline{\mathcal{G}}  \setminus \mathcal{G})^\ast(n)[2n] \rightarrow M(\mathcal{G})[1]
$$
where $n:= dim \overline{\mathcal{G}}$, see \cite[page 197]{VV}.
\newline
By proposition \ref{TheoWnderfulCompProperties}, $\overline{G}$ admits a cell decomposition.
Therefore by theorem \ref{Leray Hirsch for Voevodsky motives}, $M^c(\overline{\mathcal{G}})$ is mixed Tate. So to prove the theorem it is enough to show $M^c({\overline{\mathcal{G}} \setminus \mathcal{G}})$ is mixed Tate.
\newline
Let's now look at the geometry of the closures of $(G \times G)$-orbits. As it is mentioned in proposition \ref{TheoWnderfulCompProperties} a), these orbit closures could be indexed by a subset of  faces of Weyl chamber in such a way that the incidence relation between faces gets preserved. Note that by proposition \ref{TheoWnderfulCompProperties} b) the closure of these orbits also admit a cell decomposition. Thus by theorem \ref{Leray Hirsch for Voevodsky motives} the irreducible components of $\overline{\mathcal{G}}\setminus \mathcal{G}$ form a mixed Tate configuration. Now lemma \ref{LemMixedArrangement} implies that $M^c(\overline{\mathcal{G}}\setminus \mathcal{G})$ is mixed Tate. 
\newline
Now assume that $G$ is a reductive algebraic group. We assume that $Z(G)$ is connected. Note that since $G$ is reductive $Z:=Z(G)$ is a torus. Let $\mathcal{G}'$ be the associated $G^{ad}$-bundle. By the above statements we know that $M(\mathcal{G}')$ is mixed Tate.  
Notice that any torus bundle is locally trivial for the Zariski topology by the theorem Hilbert90. Take a toric compactification $\overline{Z}$ of $Z$ and embed $\mathcal{G}$ into $\mathcal{Z} := \mathcal{G} \times^Z \overline{Z}$, which is a toric fibration over $\mathcal{G}'$. Now the irreducible components of the complement of $\mathcal{G}$ in $\mathcal{Z}$ are toric fibrations over $\mathcal{G}'$. Since fibers are toric (and hence cellular) and $M(\mathcal{G}')$ is mixed Tate therefore by theorem \ref{Leray Hirsch for Voevodsky motives} we argue that these irreducible components form a mixed Tate configuration and we may argue as above.
\end{proof}

The assumptions that the $G$-bundle $\CG$ is locally trivial for Zariski topology and also $\charakt k=0$ may look restrictive.  As we will see below these assumptions are not necessary when $X$ is geometrically cellular. Before proving this let us state the following lemma.

\begin{lemma}\label{ReuctiveGroup}
Let $G$ be a connected reductive group over $k$, then the motive associated to $G$ is geometrically mixed Tate.
Furthermore if $G$ is a split reductive group then $M(G)$ is mixed Tate.
\end{lemma}
\bigskip
\begin{proof}
Assume that $k$ is algebraically closed. Let $T$ be a maximal split torus in $G$ of rank $r$. We consider the projection $\pi:G \rightarrow G/T$ and view $G$ as a $T$-bundle over $G/T$. Consider $\BP^r_k$ as a compactification of $T$ and let $\overline{\mathcal{T}} := G \times ^{T}\BP^r_k $ be the associated projective bundle over $G/T$.  By projective bundle formula 
$$
M(\overline{\mathcal{T}})= M(\BP^r_k ) \otimes M(G/T),
$$
see \cite[Theorem 15.12]{MVW}.
On the other hand $B = T \ltimes U$, where $B$ is a Borel subgroup of $G$ containing $T$ and $U$ is the unipotent  part of $B$. Notice that, as a variety, $U$ is isomorphic to an affine space over $k$. Since the fibration $G \rightarrow G/B$ is the composition of $G \rightarrow G/T$ and $U$-fibration $G/T \rightarrow G/B $, we deduce by proposition \ref{motive-general-flag-var} that $M(\overline{\mathcal{T}})$ is pure Tate. Finally, in the similar way as in the proof of theorem \ref{motive-of-G-bundle1}, one can see that the complement of $G$ in $\ol \CT$ form a mixed Tate configuration and conclude that $M(G)$ is mixed Tate. The second part of the lemma is similar, only one doesn't require to pass to an algebraic closure.
\end{proof}

\begin{theorem}\label{motive-of-G-bundle1b}
Let $G$ be a connected reductive group over $k$. Let $\cG$ be a $G$-bundle over an irreducible variety $X\in \CO b(Sm_k)$. Suppose in addition that $X$ is geometrically cellular, then $M(\cG)$ is geometrically mixed Tate.
\end{theorem}
\begin{proof}
We may assume that $k$ is algebraically closed. Let 
$$
\emptyset= X_{-1} \subset X_0 \subset ... \subset X_n=X 
$$
be a cell decomposition for $X$ where $U_i:=X_i \setminus X_{i-1}$ is isomorphic to $\mathbb{A}^{d_i}_k$.
We prove by induction on $n$. Consider the following distinguished gysin triangle:
$$
M(\cG|_{U_n}) \rightarrow M(\cG) \rightarrow M(\cG|_{X_{n-1}}). 
$$
By Raghunathan's theorem (i.e. the generalization of the Serre's conjecture about the triviality of vector bundles over an affine space) the restriction of $\cG$ to $U_n$ is trivial over $U_n$, see \cite{Raghunathan2}. Therefore $M(\CG|_{U_n})$ is mixed Tate. On the other hand $M(\cG|{X_{n-1}})$ is mixed Tate by induction hypothesis.
\end{proof}

Recall that the Voevodsky's theory of motives over perfect field $k$ can be established for the schemes over more general base $S$. We refer the reader to the following article of Voevodsky \cite{VVI}, or to the article \cite{CD} of Cisinski and Deglise. They construct the triangulated category of mixed motives $DM(S)$. Here $S$ is any locally noetherian scheme of finite dimension. This category is constructed from the category of Nisnevich sheaves with transfers over $X$. For the details about the category of mixed Tate motives over a number ring we refer to \cite{Sch}.\\

Let $\mathcal{A}$ be a henselian discrete valuation ring with perfect field of fractions $K$ and perfect residue field $k$ of characteristic $p > 0$.  Let $G$ be a connected reductive algebraic group over $K$.
Bruhat and Tits have associated to $G$ certain smooth affine $\CA$-group schemes $P$ with generic fiber $P_K = G$ known as \emph{parahoric group schemes}, see \cite{B-TII}. 

\begin{proposition}\label{motive of parahoric group}

Keep the above notation. Assume that the maximal torus of $G$ splits over an unramified extension of $K$. Let $R_u(P_k)$ denote the unipotent radical of the special fiber $P_k$. Assume further that $M(R_u(P_k))$ is mixed Tate. Then the motive $M(P)$ becomes mixed Tate over an \'etale covering of $\mathcal{A}$.   
\end{proposition}

\begin{proof} We consider the six functors introduced in \cite[section 1]{CD}, for the situation that:
$$
i: Speck \rightarrow Spec\mathcal{A},
$$
$$
j: SpecK \rightarrow Spec\mathcal{A}.
$$
Then we have the following distinguished triangle:
$$
j_{!}j^\ast \rightarrow id \rightarrow i_\ast i^\ast
$$
in $DM_{gm}(\mathcal{A})$. The generic fiber $P_K$ of $P$ is simply $G$, hence by lemma \ref{ReuctiveGroup} we know that $M(P_K)$ is geometrically mixed Tate. On the other hand since $G$ splits over a tamely ramified extension of $K$, we may argue by \cite{McN} that the special fiber $P_k$ has a Levi subgroup $\mathcal{L}$, in particular  $\mathcal{L} \times R_u(P_k) \rightarrow P_k$ is an isomorphism. Notice that the unipotent radical $R_u(P_k)$ is isomorphic to an affine space. Choose a finite extension of $k$ that splits $\mathcal{L}$. Since motive of the levi component becomes mixed Tate over this extension, thus also the motive associated to the special fiber $P_k$. Since $k$ is perfect and $\cA$ is henselian, this extension gives an \'etale cover of $\cA$. Now the assertion follows from the localization property, see \cite{CD} or \cite{Sch}.

\end{proof}

\section{Filtration on the motive of G-bundles}\label{Filtration on the motive of G-bundles}

\bigskip

In section \ref{SectMotiveofG-bundles} we studied the motive of a $G$-bundle over base scheme $X$, when the associated motive $M(X)$ is geometrically mixed Tate. In the sequel we produce some kind of filtration for the motive of $G$-bundles over more general base schemes in terms of incidence relation between faces of a convex body. Let us first recall the filtration on the motive of a split torus bundle constructed by B. Kahn and A. Huber. In \cite{H-K}, B. Kahn and A. Huber, as an application of the theory of \emph{slice filtration}, produce a filtration (in the triangulated category $DM_{gm}^{eff}(k)$) on the motive of a $T$-bundle $\mathcal{T}$, for a split torus $T$. Also they use this filtration to study motives of split reductive groups. Let us briefly recall their construction. 
\newline
Let $\mathcal{T}$ denote a principal $T$-bundle over a smooth variety $X \in \CO b(Sch_k)$, where $T$ is a split torus of rank $r$. Let $\Xi:= Hom(\mathbb{G}_m , T)$ be the cocharacter group. Then one can produce the following diagram of distinguished triangles in $DM_{gm}^{eff}(k)$ 
\bigskip

\[ \xygraph{
!{<0cm,0cm>;<1cm,0cm>:<0cm,1cm>::}
!{(3.5,0) }*+{M(\CT)}="a"
!{(0.5,0) }*+{\nu^{\geqslant 1}_X M(\CT)}="b"
!{(-2.5,0) }*+{\nu ^{\geqslant2}_X M(\CT)}="c"
!{(-5,0) }*+{\nu^{\geqslant p}_X M(\CT)}="m"
!{(-8,0) }*+{\nu^{\geqslant p+1}_X M(\CT)}="d"
!{(2.5,-0.7) }*+{[1]}="s"
!{(-0.5,-0.7) }*+{[1]}="t"
!{(-6,-0.7) }*+{[1]}="u"
!{(-2.5,-1) }*+{...}="l"
!{(3.5,-2) }*+{\lambda_0(X,T)}="f"
!{(0.5,-2) }*+{\lambda_1(X,T)}="g"
!{(-5,-2) }*+{\lambda_p(X,T)}="h"
"c":"b":"a"
"d":"m"
"a":"f" "b":"g" "m":"h" 
"f":"b" "g":"c" "h":"d"
}  \] \\

\bigskip

where $\lambda_p(X,T):= M(X)(p)[p] \otimes \Lambda^p(\Xi)$ for $0 \leq  p \leq  r$. Note that $M(\CT) \cong \nu_X^{\geqslant 0}M(\CT)$, $0 = \nu_X^{\geqslant r+1}M(\CT)$. \\
For more details on the construction of relative slice filters $\nu_X^{\geqslant p}M(\mathcal{T})$ see \cite[section 8]{H-K}.

\bigskip

Now using the method we introduced in section \ref{SectMotiveofG-bundles} we construct a nested filtration on the motive of a $G$-bundle.

\bigskip

Let $\cG$ be a $G$-bundle over $X$, where $G$  is a linear algebraic group. Let $\overline{G}$ be a compactification of $G$. Suppose that $D:= \overline{G} \setminus G$ form a mixed Tate configuration $D=\cup_{i=1}^m D_i$, such that $D^J:= \cap_{i \in J} D_i$ is either irreducible or empty  for any $J \subseteq \{ 1,\dots,m\}$. Assume that there exist a polytope whose faces correspond to those subsets $J\subseteq \{1,...,m\}$ such that $D^J$ is non empty (with face relation: $\CF_2$ is a face of $\CF_1$ if we have the inclusion $J_2\subseteq J_1$ of the corresponding sets). Let $\CP$ be the dual of this polytope. For each face $\cF$ of $\CP$, let $D_{\cF}$ denote the associated subvariety of $D$ regarding the above correspondence. For each $1 \leq r \leq m$, let $Q_r$ be the set consisting of all faces in $\CP$ of codimension $r$. Let $\partial \cF$ denote the boundary of $\CF$, i.e. $\partial \cF:= \{\cF \cap \mathcal{T} | \mathcal{T} \in Q_1\}\setminus\{\cF\}$.\\
Let $\overline{\cG}$ denote  the compactification $\cG \times^G \overline{G}$ of $\cG$ and let $\mathcal{D}_{\cF}$ be the associated $D_{\cF}$-fibration over $X$. We may form the following nested filtration on $M^c(\cG)$ by the distinguished triangles, indexed by codimension $r$ and faces $\cF\in Q_r$
 
\begin{equation}\label{Filt1}
\CD
M^c(\overline{\cG}\setminus \cG) \rightarrow M^c(\overline{\cG}) \rightarrow M^c(\cG)\\
\vdots\\
M^c (\bigcup_{ \cF \in Q_{r+1}} \mathcal{D}_{\cF}) \rightarrow M^c (\bigcup_{\cF \in Q_{r}} \mathcal{D}_{\cF}) \rightarrow \oplus_{\cF \in Q_r} M^c(\mathcal{D}_{\cF} \setminus \bigcup_{\cF' \in \partial \cF} \mathcal{D}_{\cF'}),\\
\vdots\\
\endCD
\end{equation}

and for each $\cF\in Q_r$ the triangle
$$
\CD
M^c (\bigcup_{\cF' \in \partial \cF} \mathcal{D}_{\cF'})) \rightarrow M^c ( \mathcal{D}_{\cF}) \rightarrow  M^c(\mathcal{D}_{\cF} \setminus \bigcup_{\cF' \in \partial \cF} \mathcal{D}_{\cF'}),
\endCD
$$

is the first line of a nested filtration obtained by replacing $\CP$ by $\cF$.

Note that this filtration is particularly interesting when $\mathcal{D}_{\mathcal{F}}$ is a cellular fibration. In this situation we may apply theorem \ref{Leray Hirsch for Voevodsky motives} to compute $ M^c(\mathcal{D}_{\cF})$. Let us recall two of such cases.

\begin{example}\label{TorusBundle}
Let $T$ be a split torus of rank $n$ and let $\mathcal{T}$ be a $T$-bundle over $X$. Consider a torus embedding of $T$ into the projective space $\mathbb{P}^n$ associated with the standard n-simplex $\Delta^n$. So we put $P:=\Delta^n$ in the above filtration. Note that in this case for each face $\cF \in \Delta^n$, $\mathcal{D}_{\cF}$ is a projective bundle and hence one may use the projective bundle formula \cite[Thm 15.11]{MVW} to compute $M^c(\mathcal{D}_{\cF})$. In particular when $M^c(X)$ is mixed Tate, using the above filtration, one may prove recursively that $M^c(\mathcal{T})$ is mixed Tate.
\end{example}

\begin{example}\label{ReductiveBundle}
Let $G$ be a semi-simple group of adjoint type and $\overline{G}$ its wonderful compactification. In this case the polytope $\CP$ is the one in proposition \ref{TheoWnderfulCompProperties}. Recall that for every face $\cF$ of $\CP$, $D_{\cF}$ admits a cell decomposition. Let us mention that for any regular compactification (see \cite{Brion} for details) of $G$ and any vertex $\cF$, $D_{\cF}$ is isomorphic to $G/B \times G/B$ and in particular $\mathcal{D}_{\cF}$ is a cellular fiberation.
\end{example}

--\textit{The case of 1-motives,} In practice it might happen that the motive of the base variety $X$ is far from being mixed Tate. Already it can happen for the case of 1-motives. Recall that the motive $M(C)$ of a curve $C$ decomposes
in $DM_{gm}^{eff}(k)\otimes \mathbb{Q}$ as follows
\begin{equation}\label{MotiveofCurves}
M(C) = M_0(C)\oplus M_1(C)\oplus M_2(C),
\end{equation}

where $M_i(C) := Tot LiAlb^{\mathbb{Q}}(C)[i]$. For the definition of $LiAlb^{\mathbb{Q}}(C)$ and detailed explanation of the theory we refer to section 3.12 of \cite{BaV-Kah}.\\

Let $C$ be a smooth projective curve. In the next example we consider the special case when the base scheme $X$ is a relative curve $C_S$ over $S$. 

\begin{example}\label{motive-of-G-bundle2}
Let $G$ be a reductive group over $k$. Let $\mathcal{G}$ be a $G$-bundle over $C$ and $\CG_s$ the associated $G^s$-bundle. Fix a closed point $p$ of $C$ and set $\dot{C}:= C \setminus \{p\}$. 
Assume that $\charakt k$ does not divide the order of $\pi_1(G)$ therefore by the weel-known theorem of Drinfeld and Simpson \cite{Drin-Simp} we may take a finite extension of $k$ which simultaneously trivializes the restriction of $\mathcal{G}_s$ over $\dot{C}$ and the fiber over $p$. Therefore we obtain the following distinguished triangle
$$
M(G^s \times \dot{C}_{k'})\rightarrow M(\mathcal{G}_{s,k'})\rightarrow M(G^s\times k')(n)[2n],
$$
and by the Kunneth theorem
$$
M(G^s) \otimes M(\dot{C}_{k'})\rightarrow M(\mathcal{G}_{s,k'})\rightarrow M(G^s\times k')(n)[2n].
$$
Let us assume that $G$ has a connected center. Since $G$ is reductive $Z:=Z(G)^{\circ}$ is a split torus. We may now apply either of the filtrations in example \ref{TorusBundle} or the relative slice filtration (see the first paragraph of this section) to the torus bundle $\mathcal{G} \rightarrow \mathcal{G}_s$. 
For instance from the latter filtration we get the following

\begin{enumerate}
\item[i)] A filtration $\{ \varphi_i : M_i \rightarrow M_{i-1} \}_{i \in \mathbb{N}}$ where $M_i:= \nu^{\geq i}_{\mathcal{G}_{s,k'}}M(\mathcal{G}_{k'})$. In particular $M_0=M(\mathcal{G}_{k'})$ and $M_r=0$ for $r>\rk Z(G)^{\circ}$ 
\item[ii)]The following sort of distinguished triangles
$$
M(\dot{C}_{k'}) \otimes M(G^s) \rightarrow M(\mathcal{G}_{s,k'}) \rightarrow M(G^s) \otimes M(k')(n)[2n]
$$
$$
M_{i+1} \rightarrow M_i \rightarrow M(\mathcal{G}_{s,k'})(i)[i] \otimes F_i,
$$ 
where $F_i$ be the $i$-th vedge power of the group $\Xi:=Hom(\mathbb{G}_m, Z)$.
\end{enumerate}

\end{example}

At the end it may look worthy to state the corresponding fact in the K-ring $K_0(Var_S)$ of varieties over $S$.\\ 
Recall that for a fibration $X\to Y$ with fiber $F$, locally trivial in Zariski topology, one has $[X]=[Y] . [F]$ where $[.]$ denotes a class in $K_0(Var_k)$, see \cite{Seb} (also see Guillet and Soul\'e \cite{GS}).  It can be shown that a similar fact (under certain assumption on the characteristic of $k$) holds for a $G$-bundle $\CG$ over the curve $C$ defined over an algebraically closed field $k$ (note however $\CG$ is not necessarily Zariski locally trivial).

\begin{proposition}\label{charcondition}
Let $G$ be a reductive group over $k$. Assume that $\charakt k$ does not divide the order of the fundamental group $\pi_1(G)$. For a $G$-bundle $\CG$ over a relative curve $C_S$ the class $[\cG]-[G\times _S C_S]$ in the Grothendieck ring of varieties $K_0(Var_S)$ lies in the kernel of the natural morphism $K_0(Var_S)\rightarrow K_0(Var_{S'})$ induced by an \'etale morphism $S'\to S$. In particular when $k$ is algebraically closed and $S=\Spec k$ then $[\CG]=[G].[C]$. \\
\end{proposition}

\begin{proof}
According to the theorem Hilbert 90 torus bundles are locally trivial for Zariski topology, thus by the result of Guillet and Soul\'e \cite{GS}, which we mentioned above, we can reduce to the case that $G=G^s$ is semi-simple. Then the proposition follows from the theorem of Drinfeld and Simpson (cf. example \ref{motive-of-G-bundle2}).
\end{proof}

%
%

{\small

\vfill

\begin{minipage}[t]{0.5\linewidth}
\noindent
Somayeh Habibi\\
Dipartimento di Matematica\\ 
``Federigo Enriques"\\ 
Universita degli Studi di Milano\\
Via C.Saldini,50 \\
20133 Milano
\\ Italy
\\[1mm]
\end{minipage}
\begin{minipage}[t]{0.45\linewidth}
\noindent
M. Esmail Arasteh Rad\\
Universit\"at M\"unster\\
Mathematisches Institut \\
Einsteinstr.~62\\
48149 M\"unster
\\ Germany
\\[1mm]

\end{minipage}


\begin{thebibliography}{GHKR2}
\addcontentsline{toc}{section}{References}

\bibitem[B-K]{BaV-Kah} L.\ Barbieri Viale, B.\ Kahn: \emph{On the derived category of 1-motives} http://arxiv.org/abs/1009.1900

\bibitem[Big]{Big} S.\ Biglari, Motives of Reductive Groups, Thesis, Leipzig 2004.

\bibitem[Br]{Brion} M.\ Brion: \emph{The behaviour at infinity of the Bruhat decomposition}. \emph{Comment.\ Math.\ Helv.} 73(1998),137-174.

\bibitem[BT]{B-T} F.\ Bruhat, J.\ Tits: Groupes r\'eductifs sur un corps local. Inst. Hautes \'Etudes Sci. Publ. Math. 41
(1972), 5–251.

\bibitem[BTII]{B-TII} F.\ Bruhat, J.\ Tits: Groupes r\'eductifs sur un corps local. II. Sch\'emas en groupes. Existence d\'une donn\'e radicielle valu\'ee. Inst. Hautes \'Etudes Sci. Publ. Math. 60 (1984), 197–376.


\bibitem[CD]{CD} D.\ C.\ Cisinski and F.\ D´eglise: \emph{Triangulated categories of motives}. Preprint, 2010.
\bibitem[C-F]{Col-Ful} A. \ Collino and W. \ Fulton, \emph{Intersection rings of spaces of triangles}, M\'em.
Soc.~Math.~France~(N.S.) 38 (1989), 75–117.

\bibitem[CGM]{CGM} V.\ Chernousov, S.\ Gille, A.\ Merkurjev: \emph{Motivic decomposition of isotropic projective homogeneous varieties}, Duke Math. J. 126 (2005), no. 1, 137–159
 
\bibitem[DP]{Con-Pro} C.\ De Concini, C.\ and Procesi: \emph{Complete symmetric varieties}, Invariant theory (Montecatini,1982), volume 996 of Lecture Notes in Math., pages 1–44, Springer, Berlin, 1983.

\bibitem[DPII]{DeConciniProcesi} C.\ DeConcini and C.\ Procesi: \emph{Complete symmetric varieties}, Lecture Notes in Math. 996, Springer, 1973, 1-44.

\bibitem[DS]{Drin-Simp} V.\ Drinfeld and C.\ Simpson: \emph{B-structures on G-bundles and local triviality},
 Math.-Res.-Lett. 2 (1995) 823-829.
 
\bibitem[EJ]{EvJo} S.\ Evens, B.\ Jones: \emph{On the Wonderful Compactification}, preprint

\bibitem[GS]{GS} H.\ Guillet and C.\ Soul\'e: \emph{Descent, motives and K-theory}. J.\ reine angew.\ Math., 4, (1996).

\bibitem[HK]{H-K} A.\ Huber , B.\ Kahn: \emph{The slice filtration and mixed Tate motives}, \emph{Compos.\ Math.} 2006;142(4):907-936.

\bibitem[Kah]{BKahn} B.\ Kahn : \emph{Motivic Cohomology of smooth geometrically cellular varieties}, Algebraic K-theory,
Seattle, 1997, Proceedings of Symposia in Pure Mathematics, vol. 67 (American Mathematical Society, Providence, RI, 1999), 149–174. 

\bibitem[Kar]{Kar} N.\ A.\ Karpenko: \emph{Cohomology of relative cellular spaces and of isotropic flag varieties}, Algebra i Analiz 12 (2000), no. 1, 3–69. MR 2001 c:14076

\bibitem[Kir]{Kiritchenko}  V.\ Kiritchenko: \emph{On intersection indices of subvarieties in reductive groups} Moscow Mathematical Journal, 7 no.3, 489-505, (2007).


\bibitem[Lam]{LAM} T.\ Y.\ Lam: \emph{Serre's Problem on Projective modules}, Springer Monographs in Mathematics, (2006).

\bibitem[LS]{Las-Sor} Y.\ Laszlo and C.\ Sorger: \emph{The line bundles on the moduli of parabolic G-bundles
over curves and their sections}. \emph{Ann. Sci. Ecole Norm. Sup. (4)}, 30(4):499-525, (1997). 

\bibitem[McN]{McN} G.\ J.\ McNinch: \emph{Levi decompositions of a linear algebraic group}, to appear: Transformation Groups (Morozov centennial issue). 

\bibitem[MVW]{MVW} C.\ Mazza, V.\ Voevodsky, C.\ A.\ Weibel. \emph{Lecture notes on motivic cohomology}, Clay mathematics monographs, v.2., (2006).

\bibitem[Rag1]{Raghunathan1} M.\ S.\ Raghunathan: \emph{Principal bundles on affine space and bundles on the projective line} \emph{Mathematische Annalen}(1989) Volume 285, Number 2, 309-332.

\bibitem[Rag2]{Raghunathan2} M.\ S.\ Raghunathan: \emph{Principal bundles on affine space}, in C. P. Ramanujam—a tribute, pp. 187–206, Tata Inst. Fund. Res. Studies in Math. 8 (1978).

\bibitem[Ren]{Renner} L.\ E.\ Renner: \emph{An Explicit Cell Decomposition of the Wonderful Compactification of a Semisimple Algebraic Group}, \emph{Canadian Mathematical Bulletin}, 46 (1): 140-148, (2003).

\bibitem[Sch]{Sch} J.\ Scholbach: \emph{Mixed Artin-Tate motives over number rings}, \emph{Journal of Pure and Applied Algebra} 215, 2106–2118,  (2011).

\bibitem[Seb]{Seb}
J.\ Sebag: \emph{Int\'egration motivique sur les sch\'emas formels}, Bull. Soc.
Math. France 132, no. 1, 1-54, (2004).

\bibitem[Tim]{Tima} D.\ Timashev, \emph{Equivariant compactifications of reductive groups}, Sb. Math. 194, no. 3-4, 589–616, (2003).

\bibitem[VSF]{VV}V.\ Voevodsly, A.\ Suslin, E.M.\ Friedlander: \emph{Cyles, Transfers, and Motivic Homology Theories}, Princeton university press, (2000)


\bibitem[Voe1]{VVI} V.\ Voevodsky: \emph{Motives over simplicial schemes}: J. K-Theory 5,
no. 1, 1–38, (2010).

\bibitem[Voe2]{VVII} V.\ Voevodsky: \emph{Homology of Schemes}, Selecta Mathematica, New Series, 2(1):111–153, (1996).
\end{thebibliography}
\end{document}